\documentclass[11pt,a4paper,oneside]{amsart}
\newcounter{commentcounter}

\usepackage{amsmath} 
\usepackage{amssymb} 
\usepackage{amsthm} 
\usepackage{stmaryrd} 
\usepackage[english]{babel} 
\usepackage[font=small,justification=centering]{caption} 
\usepackage[nodayofweek]{datetime}
\usepackage[shortlabels]{enumitem} 
\usepackage[T1]{fontenc} 
\usepackage[utf8]{inputenc} 
\usepackage{ifthen} 
\usepackage{mathabx} 
\usepackage{mathtools} 
\usepackage[dvipsnames]{xcolor} 
\usepackage[pdftex,  colorlinks=true, pagebackref=true]{hyperref} 
    \hypersetup{urlcolor=RoyalBlue, linkcolor=RoyalBlue,  citecolor=black}
\usepackage{setspace} 
\usepackage{tikz-cd} 
\usepackage{xfrac} 
\usepackage[capitalize]{cleveref} 
\usepackage{booktabs}
\renewcommand*{\backref}[1]{}
\renewcommand*{\backrefalt}[4]
{
    \ifcase #1
        No citation in the text.
    \or
        Cited on Page #2.
    \else
        Cited on Pages #2.
    \fi
}

\makeatletter
\@namedef{subjclassname@1991}{Mathematical subject classification 1991}
\@namedef{subjclassname@2000}{Mathematical subject classification 2000}
\@namedef{subjclassname@2010}{Mathematical subject classification 2010}
\@namedef{subjclassname@2020}{Mathematical subject classification 2020}
\makeatother

\newtheorem{thm}{Theorem}[section]
\newtheorem{lemma}[thm]{Lemma}

\newtheorem{prop}[thm]{Proposition}

\theoremstyle{definition}

\newtheorem{remark}[thm]{Remark}

\theoremstyle{plain}

    \newtheoremstyle{TheoremNum}
        {\topsep}{\topsep} 
        {\itshape} 
        {-0.25cm} 
        {\bfseries} 
        {.} 
        { }  
        {\thmname{#1}\thmnote{ \bfseries #3}}
    \theoremstyle{TheoremNum}

\newcommand*{\claimproofname}{My proof}

\DeclareMathOperator{\tr}{\mathrm{tr}}

\DeclareMathOperator{\sym}{Sym}

\newcommand{\GL}{\mathrm{GL}}

\def\Z{\mathbb{Z}}

\newcommand{\QQ}{\mathbb{Q}}

\usepackage{tikz}
\usetikzlibrary{arrows,quotes}
\tikzstyle{blackNode}=[fill=black, draw=black, shape=circle]

\newcommand{\atA}{\tilde{\mathtt{A}}} 
\newcommand{\atX}{\widetilde{\mathtt{X}}}
\newcommand{\aX}{\mathtt{X}}
\newcommand{\aB}{\mathtt{B}}
\newcommand{\aC}{\mathtt{C}}
\newcommand{\atB}{\widetilde{\mathtt{B}}}
\newcommand{\atC}{\widetilde{\mathtt{C}}}
\newcommand{\atD}{\widetilde{\mathtt{D}}}
\newcommand{\atE}{\widetilde{\mathtt{E}}}
\newcommand{\atF}{\widetilde{\mathtt{F}}_{4}}
\newcommand{\atG}{\widetilde{\mathtt{G}}_{2}}
\newcommand{\atI}{\widetilde{\mathtt{I}}_{1}}

\title{Profinite rigidity of affine Coxeter groups}

\author{Samuel M. Corson, Sam Hughes, Philip Möller, and Olga Varghese}
\date{\today}

\address{Samuel M. Corson\\
Matematika Saila\\
UPV/EHU\\
Sarriena s/n\\
48940 Leioa--Bizkaia (Spain)
}
\email{sammyc973@gmail.com}

\address{Sam Hughes\\
Mathematical Institute\\
Andrew Wiles Building\\ 
Observatory Quarter\\ 
University of Oxford\\ 
Oxford\\ 
OX2 6GG (United Kingdom)}
\email{sam.hughes@maths.ox.ac.uk}

\address{Philip M\"oller\\
Department of Mathematics\\
University of M\"unster\\ 
Einsteinstra\ss e 62\\
48149 M\"unster (Germany)}
\email{philip.moeller@uni-muenster.de}

\address{Olga Varghese\\ Institute of Mathematics, Heinrich-Heine-University Düsseldorf, Universitätsstra{\upshape{\ss}}e 1, 40225, Düsseldorf (Germany)}
\email{olga.varghese@hhu.de}



\begin{document}
	
\pagenumbering{arabic}
	
	\begin{abstract}
	We prove that affine Coxeter groups are profinitely rigid.
\end{abstract}

\maketitle

\section{Introduction}
For a group $G$  we denote by $\mathcal{F}(G)$ the set of isomorphism classes of finite quotients of $G$. 
A group $G$ is called \emph{profinitely rigid relative to a class of groups $\mathcal{C}$} if $G\in \mathcal{C}$ and for any group $H$ in the class $\mathcal{C}$ whenever $\mathcal{F}(G)=\mathcal{F}(H)$, then $G\cong H$. A finitely generated residually finite group $G$ is called \emph{profinitely rigid} if $G$ is profinitely rigid among all finitely generated residually finite groups. 

\begin{thm}
\label{IrrAffineCoxeterProfinitelyRigid}
Affine Coxeter groups are profinitely rigid.
\end{thm}

Given a finite graph $\Gamma$ with the vertex set $V(\Gamma)$, the edge set $E(\Gamma)$ and an edge-labeling $m\colon E(\Gamma)\to\mathbb{N}_{\geq 3}\cup\left\{\infty\right\}$, the associated \emph{Coxeter group} $W_\Gamma$ is given by the presentation
\begin{gather*}
W_{\Gamma}=\left\langle V(\Gamma)\ \middle\vert \begin{array}{l} 
v^2\text{ for all }v\in V(\Gamma), (vw)^2 \text{ if }\left\{v,w\right\}\notin E(\Gamma),\\ (vw)^{m(\left\{v,w\right\})}\text{ if } \left\{v,w\right\}\in E(\Gamma) \text{ and } m(\left\{v,w\right\})<\infty\end{array} \right\rangle.
\end{gather*}

The Coxeter groups associated to the graphs in \Cref{affineCoxeterGraphs} are precisely the \emph{irreducible affine Coxeter groups}. More generally, a Coxeter group $W_\Gamma$ is \emph{affine} if $\Gamma$ is a disjoint union of those graphs.
It was shown in \cite{MollerVarghese2023} that irreducible affine Coxeter groups are profinitely rigid relative to the class consisting of all Coxeter groups, our main result generalises this.  Other work on profinite rigidity of Coxeter groups can be found in \cite{BridsonConderReid2016,BridsonMcReynoldsReidSpitler2021,CorsonHughesMollerVarghese2023,SantosRegoSchwer2022}.
\begin{figure}[h]
	\begin{center}
	\captionsetup{justification=centering}
		\begin{tikzpicture}[scale=0.8, transform shape]
			\draw[fill=black]  (0,0) circle (2pt);
			\draw[fill=black]  (1,0) circle (2pt);
			\draw (0,0)--(1,0);
			\node at (0.5,0.2){$\infty$};
			\node at (-0.9, 0) {$\atI$};	
            \draw[fill=black] (0,-1.5) circle (2pt);
            \draw[fill=black] (1,-1.5) circle (2pt);
            \draw (0,-1.5)--(1,-1.5);
            \draw[fill=black] (2,-1.5) circle (2pt);
            \draw[fill=black] (3,-1.5) circle (2pt);
            \draw[dashed] (1,-1.5)--(2, -1.5);
            \draw (2,-1.5)--(3,-1.5);
            \draw[fill=black] (1.5, -0.8) circle (2pt);
            \draw (0,-1.5)--(1.5, -0.8);
            \draw (3,-1.5)--(1.5, -0.8);
            \node at (-0.9, -1.5) {$\underset{n\geq 2}{\atA_n}$};	
            \draw[fill=black] (0,-3) circle (2pt);
            \draw[fill=black] (1,-3) circle (2pt);
            \node at (0.5,-2.8) {$4$};
            \draw (0,-3)--(1,-3);
            \draw (1,-3)--(2,-3);
            \draw[dashed] (2,-3)--(3,-3);
            \draw[fill=black] (3,-3) circle (2pt);
            \draw[fill=black] (2,-3) circle (2pt);
            \draw[fill=black] (4, -2.3) circle (2pt);
            \draw[fill=black] (4, -3.7) circle (2pt);
            \draw (3, -3)--(4,-2.3);
            \draw (3,-3)--(4, -3.7);
            \node at (-0.9, -3) {$\underset{n\geq 3}{\atB_n}$};	
            \draw[fill=black] (0,-4.5) circle (2pt);
            \draw[fill=black] (1,-4.5) circle (2pt);
            \draw (0,-4.5)--(1, -4.5);
            \node at (0.5, -4.3) {$4$};
            \draw[fill=black] (2, -4.5) circle (2pt);
            \draw (1, -4.5)--(2,-4.5);
            \draw[dashed] (2,-4.5)--(3,-4.5);
            \draw[fill=black] (3,-4.5) circle (2pt);
            \draw[fill=black] (4,-4.5) circle (2pt);
            \draw (3,-4.5)--(4,-4.5);
            \node at (3.5, -4.3) {$4$};
            \node at (-0.9, -4.5) {$\underset{n\geq 2}{\atC_n}$};
            \draw[fill=black] (0,-5.3) circle (2pt);
            \draw[fill=black] (0,-6.7) circle (2pt);
            \draw[fill=black] (1,-6) circle (2pt);
            \draw (0,-5.3)--(1,-6);
            \draw (0,-6.7)--(1,-6);
            \draw[fill=black] (2,-6) circle (2pt);
            \draw (1,-6)--(2,-6);
            \draw[dashed] (2,-6)--(3,-6);
            \draw[fill=black] (3,-6) circle (2pt);
            \draw[fill=black] (4, -5.3) circle (2pt);
            \draw[fill=black] (4, -6.7) circle (2pt);
            \draw (3,-6)--(4, -5.3);
            \draw (3,-6)--(4, -6.7);           
            \node at (-0.9, -6) {$\underset{n\geq 4}{\atD_n}$};
            \draw[fill=black] (6,0) circle (2pt);
            \draw[fill=black] (7,0) circle (2pt);
            \draw[fill=black] (8,0) circle (2pt);
            \draw[fill=black] (9,0) circle (2pt);
            \draw[fill=black] (10,0) circle (2pt);
            \draw[fill=black] (8,1) circle (2pt);
            \draw[fill=black] (8,2) circle (2pt);
            \draw (6,0)--(10,0);
            \draw (8,0)--(8,2);
            \node at (5.5,0) {$\atE_6$};
            \draw[fill=black] (6, -1.5) circle (2pt);
            \draw[fill=black] (7, -1.5) circle (2pt);
            \draw[fill=black] (8, -1.5) circle (2pt);
            \draw[fill=black] (9, -1.5) circle (2pt);
            \draw[fill=black] (10, -1.5) circle (2pt);
            \draw[fill=black] (11, -1.5) circle (2pt);
            \draw[fill=black] (12, -1.5) circle (2pt);
            \draw[fill=black] (9, -0.5) circle (2pt);
            \draw (6,-1.5)--(12, -1.5);
            \draw (9, -1.5)--(9, -0.5);
            \node at (5.5,-1.5) {$\atE_7$};
            \draw[fill=black] (6, -3) circle (2pt);
            \draw[fill=black] (7, -3) circle (2pt);
            \draw[fill=black] (8, -3) circle (2pt);
            \draw[fill=black] (9, -3) circle (2pt);
            \draw[fill=black] (10, -3) circle (2pt);
            \draw[fill=black] (11, -3) circle (2pt);
            \draw[fill=black] (12, -3) circle (2pt);
            \draw[fill=black] (13, -3) circle (2pt);
            \draw[fill=black] (8, -2) circle (2pt);
            \draw (6,-3)--(13, -3);
            \draw (8, -3)--(8, -2);
            \node at (5.5,-3) {$\atE_8$};
            \draw[fill=black] (6, -4.5) circle (2pt);
            \draw[fill=black] (7, -4.5) circle (2pt);
            \draw[fill=black] (8, -4.5) circle (2pt);
            \draw[fill=black] (9, -4.5) circle (2pt);
            \draw[fill=black] (10, -4.5) circle (2pt);
            \draw (6, -4.5)--(10, -4.5);
            \node at (8.5, -4.3) {$4$};
            \node at (5.5,-4.5) {$\atF$};
            \draw[fill=black] (6, -6) circle (2pt);
            \draw[fill=black] (7, -6) circle (2pt);
            \draw[fill=black] (8, -6) circle (2pt);
            \draw (6,-6)--(8, -6);
            \node at (7.5, -5.8) {$6$};
            \node at (5.5,-6) {$\atG$};   
		\end{tikzpicture}
	\caption{Coxeter graphs of affine type.
 }\label{affineCoxeterGraphs}
	\end{center}
\end{figure}
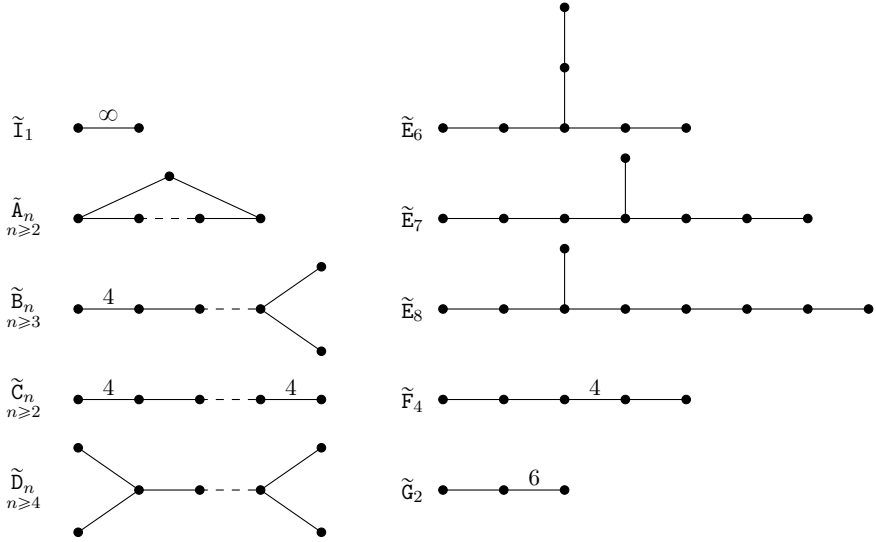 

An \emph{$n$-dimensional crystallographic group} $G$ is a discrete, cocompact subgroup of the group of isometries of the Euclidean space $\mathbb{E}^n$.
An $n$-dimensional crystallographic group $G$ always gives rise to the short exact sequence
$1\hookrightarrow \mathbb{Z}^n\hookrightarrow G \twoheadrightarrow P\twoheadrightarrow 1$
where $P$ is finite and is called the \emph{point group} of $G$. By definition, $G$ is \emph{symmorphic} if the above short exact sequence splits. Note that affine Coxeter groups are examples of symmorphic crystallographic groups.

The next proposition collects old and new profinite invariants of crystallographic groups.  A group $G$ is said to be \emph{just infinite} if $G$ itself is infinite but all proper quotients of $G$ are finite. Let $G$ and $H$ be  crystallographic groups with point groups $P_1, P_2\leqslant\GL_n(\Z)$. By definition, $P_1$ and $P_2$ are in the same $\mathbb{Q}$-class if they are conjugate in $\GL_n(\QQ)$. For a group $G$, the lattice $\mathcal{CF}(G)$ is the lattice of finite subgroups of $G$ modulo the conjugacy relation.

\begin{prop}\label{prop.invariants}
        Let $G$ be an $n$-dimensional crystallographic group and $H$ be a finitely generated residually finite group. If $\mathcal{F}(G)=\mathcal{F}(H)$, then $H$ is an $n$-dimensional crystallographic group whose point group is isomorphic to the point group of $G$. 
    In particular, if $G$ is profinitely rigid relative to the class of $n$-dimensional crystallographic groups, then $G$ is profinitely rigid in the absolute sense. Moreover, the following statements hold
    \begin{enumerate}
        \item $\mathcal{CF}(G)=\mathcal{CF}(H)$;
        \item $G$ is torsion free if and only if $H$ is torsion free;
        \item $G$ is centreless if and only if $H$ is centreless;
        \item $G$ is just infinite if and only if $H$ is just infinite;
        \item $G^{ab}\cong H^{ab}$;
        \item $G$ is symmorphic if and only if $H$ is symmorphic;
        \item the point group of $G$ is in the same $\mathbb{Q}$-class as the point group of $H$. 
    \end{enumerate}
\end{prop}
\begin{proof}
The first paragraph of the proposition is given by \Cref{rigidity_symm_tf}.  We now, prove the ``moreover''.  (1) is \Cref{LatticeProfiniteInvariant}, (2) is given by \Cref{rigidity_symm_tf}, (3) is \Cref{centre}, (4) is \Cref{justinfinite}, (5) is classical (see for example \cite{Reid2018}), (6) essentially follows from Grunewald--Zalesskii \cite{GrunewaldZalesskii2011} but we include a proof for completeness (See \Cref{rigidity_symm_tf}), and finally  (7) is due to Piwek--Popovic--Wilkes \cite[page 558]{PiwekPopovicWilkes2021}.
\end{proof}

\begin{remark}\label{4Dim}
    If follows from \Cref{prop.invariants} and \cite{PiwekPopovicWilkes2021} that every crystallographic group in dimension at most $4$ is profinitely rigid. Note that \cite{Finken1980} provides an example of an $11$-dimensional crystallographic group with point group of order $55$ which is not profinitely rigid. Further, for each prime number $p\geq 23$ there exist non profinitely rigid crystallographic groups of shape $\Z^{p-1}\rtimes\Z_p$, see \cite[Theorem 1]{Brigham1971}.  
\end{remark}

\bigskip

\subsection*{Acknowledgements} 
This work has received funding from the European Research Council (ERC) under the European Union's Horizon 2020 research and innovation programme (Grant agreement No. 850930). OV is supported by DFG grant VA~1397/2-2. 

\section{The ingredients in Proposition~{\ref{prop.invariants}}}
The following lemma is well known.

\begin{lemma}
\label{CharacterisationCrystallographicGroup}
Let $G$ be a virtually free abelian group of rank $n\geq 1$. The group $G$ is an $n$-dimensional crystallographic group if and only if $G$ does not have non-trivial finite normal subgroups.
\end{lemma}

Let $G$ be a group and $\mathcal{N}$ be the set of all finite index normal subgroups of $G$. We equip each $G/N$, $N\in\mathcal{N}$ with the discrete topology and endow $\prod_{N\in\mathcal{N}} G/N$ with the product topology. We define a map 
$$\iota\colon G\to \prod_{N\in\mathcal{N}} G/N\text{ by }g\mapsto (gN)_{N\in\mathcal{N}}.$$ The map $\iota$ is injective if and only if $G$ is residually finite. The \emph{profinite completion} of $G$, denoted by $\widehat{G}$, is defined as $\widehat{G}:=\overline{\iota(G)}$. Let $G$ and $H$ be finitely generated residually finite groups. Then $\mathcal{F}(G)=\mathcal{F}(H)$ if and only if $\widehat{G}\cong\widehat{H}$, see \cite{Dixon1982}.

\begin{lemma}
\label{FNS}
Let $G$ be a finitely generated residually finite group. Denote by $\iota\colon G\to\widehat{G}$ the canonical homomorphism. If $N\trianglelefteq G$ is a finite normal subgroup, then $\iota(N)$ is normal in $\widehat{G}$. 
\end{lemma}
Compare to the proof of Theorem 3.6 in \cite{BridsonConderReid2016}. 
\begin{proof}
Assume for contradiction that $\iota(N)$ is not normal in $\widehat{G}$. Then there exist $n\in \iota(N)$ and $g\in\widehat{G}$ such that $gng^{-1}\notin\iota(N)$. Hence the finite set $S:=\left\{gng^{-1}m\mid m\in\iota(N)\right\}$ does not include the trivial element. We know that $\iota(N)=\left\{m_1,\ldots, m_l\right\}$. Since $\widehat{G}$ is residually finite, there exists an epimorphism $\psi_k\colon\widehat{G}\twoheadrightarrow H_k$ with $H_k$ finite and $\psi_k(gng^{-1}m_k)\neq 1$ for every $k\in \{1,\ldots ,l\}$. 

Define $\psi = \psi_1\times\ldots\times\psi_l\colon\widehat{G}\to H_1\times\ldots\times H_l$ by 
$(\psi_1\times\ldots\times\psi_l)(h)=(\psi_1(h),\ldots,\psi_l(h))$. In particular, this map has finite image and $1 \notin \psi(S)$. But $\psi\circ \iota(N)$ is normal in the image $\psi\circ \iota(G)$, and $\psi\circ \iota(G) = \psi(\widehat{G})$ by \cite[Lemma 2.1]{BridsonConderReid2016}, so it is necessary that $1 \in \psi(S)$. This contradiction shows that $\iota(N)$ is normal in $\widehat{G}$.
\end{proof}

Given a group $G$ we denote by $\mathcal{CF}(G)$ the set of conjugacy classes of all finite subgroups in $G$.  We define a partial order on $\mathcal{CF}(G)$ as follows: $[A]\leq [B]$ if there exists a $g\in G$ such that $A\subseteq gBg^{-1}$.

\begin{prop}
\label{NormalFiniteCrystallographic}
Let $G$ be a finitely generated virtually free abelian group. Then, $\mathcal{CF}(G)=\mathcal{CF}(\widehat G)$.
\end{prop}
\begin{proof}
Let $G$ be a finitely generated virtually free abelian group. 
We define a map $\psi\colon\mathcal{CF}(G)\rightarrow\mathcal{CF}(\widehat{G})$ via $\psi\left([A]\right):= \left[\iota(A)\right]$.  Note $\psi$ is clearly order preserving.
 
Virtually abelian groups are finite subgroup separable by \cite[Theorem 1]{GrunewaldSegal1978}. Thus by \cite[Lemma 3.4]{CorsonHughesMollerVarghese2023} the map $\psi$ is injective. 

We follow the proof strategy used in \cite[Theorem 2.7]{KrophollerWilson93}. 
 For the surjectivity we show that a finite subgroup of $\widehat{G}$ is conjugate to a finite subgroup of $G$. Let $H$ denote a finite subgroup of $\widehat{G}$. Since $G$ is virtually free abelian, there exists a normal subgroup $A\cong \Z^n$ such that $Q:=G/A$ is finite. 
Thus, we have $H\subseteq G\cdot\widehat{A}=\widehat{A}\cdot G$.  Define $\rho\colon H\times \widehat{A}\to\widehat{A}$ where $\rho(h, a)=hah^{-1}$. Since $G$ and $\widehat{A}$ normalize $\widehat{A}$, so does $H$. Thus $\widehat{A}$ is an $H$-module, $A$ is an $H$-submodule and $\widehat{A}/A$ is an $H$-module.

Let $h\in H$. There exist elements $g_h\in G$ and $x_h\in \widehat{A}$ such that $h=x_hg_h$. The element $x_h$ is in general not uniquely determined by $h$, however, its image in $\widehat{A}/A$ is, since $G\cap \widehat{A}=A$. 

Consider the map $D\colon H\to \widehat{A}/A$ by $h\mapsto x_hA$. A computation shows that the map $D$ is a \emph{derivation}, that is, $D(h_1h_2)=D(h_1)+h_1D(h_2)$, where $h_1D(h_2)=h_1x_{h_2}h_1^{-1}A$, for $h_1,h_2\in H$.   
We claim that $H^1(H;\widehat{A}/A)=0$. 
Indeed, let $k$ denote the order of $H$, let $f\in H^1(H;\widehat{A}/A)$ denote a derivation and $g\in H$ an arbitrary element and set $x:=\sum_{h\in H}f(h)$. Now, we can compute that $gx=\sum_{h\in H}f(h)-kf(g)=x-kf(g)$.
Therefore, $kf$ equals $0$ in $H^1(H;\widehat{A}/A)$; so $kf(g)=gx'-x'$ for some $x'\in \widehat{A}/A$. Since $\widehat{A}/A$ is $k$-divisible, we can divide by $k$ and obtain $f(g)=gy-y$ for $y=x'/k$.  Thus, $f=0$.

Since $H^1(H;\widehat{A}/A)=0$, we see that $D$ is an \emph{inner derivation}, that is there exists a $b\in \widehat{A}$ such that $D(h)=hbh^{-1}b^{-1}A$ for every $h\in H$. It follows that $bhb^{-1}\in G$, since $D(h)=x_hA=hbh^{-1}b^{-1}A$, which implies $g_h bhb^{-1}\in A\subseteq G$. Hence, $bHb^{-1} \subseteq G$ as desired. This implies the surjectivity of $\psi$.
\end{proof}

\begin{prop}\label{LatticeProfiniteInvariant}
    Let $G$ be a finitely generated virtually free abelian group  and $H$ be a finitely generated residually finite group such that $\widehat{G}\cong\widehat{H}$.  Then, $\mathcal{CF}(G)=\mathcal{CF}(H)$.
\end{prop}
\begin{proof}
    Since $\widehat{G}\cong\widehat{H}$ is virtually abelian and $H\hookrightarrow \widehat{H}$ it follows that $H$ is a virtually free abelian group. Hence, by \Cref{NormalFiniteCrystallographic} we have order isomorphisms $\mathcal{CF}(G)\rightarrow\mathcal{CF}(\widehat{G})\rightarrow\mathcal{CF}(\widehat{H})\rightarrow\mathcal{CF}(H)$.  Let $\alpha$ denote the composite isomorphism and note that for any $[A]\in\mathcal{CF}(G)$ and $B\in \alpha([A])$ we have $A\cong B$.
\end{proof}

\begin{lemma}
\label{symmorphic}
Let $G$ be a crystallographic group with point group $P$. Then $G$ is symmorphic if and only if $G$ has a subgroup isomorphic to $P$.
\end{lemma}
\begin{proof}
Let $1\hookrightarrow \mathbb{Z}^n\hookrightarrow G \overset{\pi}{\twoheadrightarrow} P\twoheadrightarrow 1$ be the short exact sequence associated to $G$. If there exists a group homomorphism $\varphi\colon P\to G$ such that $\pi\circ\varphi=id_P$, then $\varphi$ is injective and therefore $G$ has a subgroup $\varphi(P)\cong P$.

For the other direction let $H\leqslant G$ be a subgroup such that $H\cong P$. Since the kernel of $\pi$ is torsion free, the map $\pi_{|H}\colon H\to P$ is injective and therefore an isomorphism since $|H|=|P|$. We define $\phi:=\pi^{-1}_{|H}$. It is straightforward to verify that $\phi$ is a section.
\end{proof}

\begin{prop}\label{rigidity_symm_tf}
    Let $G$ be an $n$-dimensional crystallographic group with point group $P$ and $H$ be a finitely generated residually finite group. If $\widehat{G}\cong\widehat{H}$, then $H$ is an $n$-dimensional crystallographic group with point group isomorphic to $P$. Moreover,
    \begin{enumerate}
        \item $G$ is symmorphic if and only if $H$ is symmorphic.
        \item $G$ is torsion free if and only if $H$ is torsion free.
    \end{enumerate}
\end{prop}
\begin{proof}
Let  $H$ be a finitely generated residually finite group with $\widehat{G}\cong\widehat{H}$.
By \cite[Proposition 2.10]{GrunewaldZalesskii2011} follows that $H$ is a virtually free abelian group of rank $n$ with quotient isomorphic to $P$. 

By \Cref{CharacterisationCrystallographicGroup} the crystallographic group $G$ does not have non-trivial finite normal subgroups, thus by \Cref{NormalFiniteCrystallographic} we know that $\widehat{G}$ does not have non-trivial finite normal subgroups. Hence, $\widehat{H}$ and therefore $H$ does not have any non-trivial finite normal subgroups either. Thus, by \Cref{CharacterisationCrystallographicGroup} we see that $H$ is an $n$-dimensional crystallographic group.

Now, \Cref{LatticeProfiniteInvariant} implies that $G$ is torsion free if and only if $H$ is torsion free. Further, $G$ has a subgroup isomorphic to $P$ if and only if $H$ has a subgroup isomorphic to $P$. Thus, by \Cref{symmorphic} we obtain that $G$ is symmorphic if and only if $H$ is symmorphic.
\end{proof}

\begin{thm}~\cite[Theorem 6]{RatcliffeTschantz2010}
\label{cryst.center}
Let $G$ be a crystallographic group. Then  $Z(G)\cong\Z^n$, where $n$ is the rank of the abelianization of $G$.
\end{thm}
 
\begin{prop}
\label{centre}
Let $G$ be an $n$-dimensional crystallographic group.
Then $\widehat{G}$ is centreless if and only if $G$ is centreless.
\end{prop}
\begin{proof}
We have $Z(G)\subseteq Z(\widehat{G})$ (see \cite[Lemma 2.1]{BridsonReidSpitler2023}). Hence, if $\widehat{G}$ is centreless, then $G$ is centreless as well. 

Now, assume that $Z(G)$ is trivial. By \Cref{cryst.center} we know that $G$  has finite abelianization, thus the commutator subgroup $[G,G]$ has finite index in $G$ and therefore $\overline{[G,G]}=\widehat{[G,G]}$ has finite index in $\widehat{G}$. It follows that $\widehat{G}^{\mathrm{ab}}$ is finite.

The profinite completion $\widehat{G}$ has a normal subgroup $N$ isomorphic to $\widehat{\mathbb{Z}^n}$ such that $\widehat{G}/N\cong P$ where $P$ is the point group of $G$. Let $m=|P|$. 

Assume for a contradiction that $\widehat{G}$ has a non-trivial centre. By \Cref{NormalFiniteCrystallographic} we know that $\widehat{G}$ does not have non-trivial finite normal subgroups, hence the torsion part of the centre of $\widehat{G}$ is trivial. Thus there exists a non-trivial $n_0\in N\cap Z(\widehat{G})$. 

Now we consider the transfer map $\tr\colon\widehat{G}\to N$ defined by Schur in \cite{Schur1902} as follows: let $g_1,\ldots,g_m$ be a set of left coset representatives of $N$ in $\widehat G$. For $g\in\widehat{G}$ and $i=1,\ldots,m$, there exists $n_i\in N$ such that $gg_i=g_jn_i$ for some $g_j$. We define $\tr(g):=n_1+\ldots+n_m$. In particular we have: $\tr(n_0)=m\cdot n_0$, thus the order of $\tr(n_0)$ is infinite. Hence, $\widehat{G}$ has an infinite abelian quotient which contradicts the fact that the abelianization of $\widehat{G}$ is finite.
\end{proof}

\begin{prop}
\label{justinfinite}
Let $G$ be a virtually free abelian group of rank $n\geq 1$ with quotient $P$ and $H$ be a finitely generated residually finite group. If $\widehat{G}\cong\widehat{H}$, then $G$ is just infinite if and only if $H$ is just infinite.
\end{prop}
\begin{proof}
By \Cref{rigidity_symm_tf} we have that $H$ is a crystallographic group.  Since $\widehat{G}\cong\widehat{H}$, 
the point groups $P$ and $P'$ are in the same $\mathbb{Q}$-class by \cite[page 558]{PiwekPopovicWilkes2021}.
A result of Ratcliffe--Tschantz \cite[Theorem 11]{RatcliffeTschantz2010} shows that a crystallographic group is just infinite if and only if the corresponding representation of the point group $P\to \GL_n(\Z)$ is $\Z$-irreducible. Moreover, by \cite[page 497]{CurtisReiner1962} we have that $\Z$-irreducibility is equivalent to $\mathbb{Q}$-irreducibility. Since $\mathbb{Q}$-irreducibility is preserved by conjugation in $\GL_n(\mathbb{Q})$, it follows that $G$ is just infinite if and only if $H$ is just infinite.  
\end{proof}

\section{Proof of Theorem~\ref{IrrAffineCoxeterProfinitelyRigid}}

The following lemma follows from \cite[Proposition 17.2.1]{Davis2008}, \cite[Theorem 3.4]{ParisVarghese2024}, and \Cref{CharacterisationCrystallographicGroup}.

\begin{lemma}
\label{AffineCoxeterGroupsCrystal}
A Coxeter group $W_\Gamma$ is crystallographic if and only if every connected component of $\Gamma$ is isomorphic to one of the graphs in \Cref{affineCoxeterGraphs}.
\end{lemma}

\begin{proof}[Proof of \Cref{IrrAffineCoxeterProfinitelyRigid}]
We first prove the result for the irreducible crystallographic Coxeter groups. 
 Let $\atX_n$ be one of the graphs in \Cref{affineCoxeterGraphs} and let $W=W_{\atX_n}$. Then $W\cong Q(\aX_n^\vee)\rtimes W_{\aX_n}$, where $Q(\aX_n^\vee)\cong\Z^n$ is the corresponding coroot lattice and $W_{\aX_n}$ is the corresponding finite Coxeter group. We denote by $Q(\aX_n)$ the corresponding root lattice and by $P(\aX_n)$ the weight lattice.  See \cite[pages 81 and 118]{Kane2001} for the definitions of these lattices and for the Coxeter graphs of type $\aX_n$.
Note that by \Cref{AffineCoxeterGroupsCrystal} the Coxeter group $W$ is an $n$-dimensional crystallographic group. 

Let $G$ be a finitely generated residually finite group such that $\widehat{W}\cong\widehat{G}$. If $n\leq 4$, then $W\cong G$ by \Cref{4Dim}.  
Now we assume that $n\geq 5$.
By \Cref{prop.invariants} it follows that $G$ is an $n$-dimensional symmorphic crystallographic group whose point group is in the same $\mathbb{Q}$-class as $W_{\aX_n}$ and $W^{\mathrm{ab}}\cong G^{\mathrm{ab}}$.  We consider two cases:

\textbf{Case 1:}  \emph{Assume that $\atX_n$ is not of type $\atB_n$ or $\atC_n$.}

Since $n\geq 5$ the Coxeter graph $\atX_n$ is of type $\atA_n, \atD_n, \atE_6, \atE_7$ or $\atE_8$. Thus the corresponding root lattice $Q(\aX_n)$ is equal to the coroot lattice $Q(\aX_n^\vee)$, see \cite[pages 102--105]{Kane2001}. Hence $W\cong Q(\aX_n)\rtimes W_{\aX_n}$.

By \cite[Theorem 1]{Feit1998}, there exists a $W_{\aX_n}$ invariant lattice $L$ such that $G\cong L\rtimes W_{\aX_n}$ and $Q(\aX_n)\subseteq L\subseteq P(\aX_n)$. Thus $W$ is a normal subgroup of $G$ of index $|L/Q(\aX_n)|$.  Note that $W^{\mathrm{ab}}\cong W_{\aX_n}^{\mathrm{ab}}\cong \Z_2 $ by \cite[Propositions 2.2 and 2.3]{MollerVarghese2023}. Since $G^{\mathrm{ab}}\cong W^{ab}\cong \Z_2 $, it follows that $L/Q(\aX_n)$ is 
trivial or is isomorphic to $\Z_2 $.

The lattice $Q(\aX_n)$ is a normal subgroup of $G$, thus $G/Q(\aX_n)\cong L/Q(\aX_n)\rtimes W_{\aX_n}$. Since $|L/Q(\aX_n)|\leq 2$, the semidirect product is indeed a direct product. Thus 
$G/Q(\aX_n)\cong L/Q(\aX_n)\times W_{\aX_n}\twoheadrightarrow L/Q(\aX_n)\times W^{\mathrm{ab}}_{\aX_n}\cong L/Q(\aX_n)\times \Z_2 .$ 
Since $G^{\mathrm{ab}}\cong W^{\mathrm{ab}}\cong \Z_2 $ we conclude that $L/Q(\aX_n)$ is trivial and therefore $G\cong W$.

\textbf{Case 2:} \emph{Assume that $\atX_n$ is of type $\atB_n$ or $\atC_n$.}

First we note that $W_{\aB_n}=W_{\aC_n}$. Further, the irreducible affine Coxeter group $W_{\atX_n}$  does not have a quotient isomorphic to $(\Z_2^2)\rtimes W_{\aB_n}$ or $\Z_4\rtimes W_{\aB_n}$ by \cite[Proposition 7.2]{Maxwell1998}.

By \cite[Theorem 1]{Feit1998}, there exist $W_{\aB_n}$ invariant lattices $L_1\subseteq L_2\subseteq L_3$ such that $|L_{i}/L_{i-1}|=2$ for $i=2,3$ and $W_{\atB_n}\cong L_l\rtimes W_{\aB_n}$, $W_{\atC_n}\cong L_k\rtimes W_{\aB_n}$ and $G\cong L_m\rtimes W_{\aB_n}$ for $k,l,m\in \left\{1,2,3\right\}$. Moreover, $L_3/L_1\cong\Z_4$ if $n$ is odd and $L_3/L_1\cong \Z_2^2 $ if $n$ is even. Thus the group $L_3\rtimes W_{\aB_n}$ has a quotient isomorphic to $\Z_2^2\rtimes W_{\aB_n}$ if $n$ is odd and  $\Z_4\rtimes W_{\aB_n}$ if $n$ is even, namely $(L_3\rtimes W_{\aB_n})/L_{1}$.

Further, the group $L_2\rtimes W_{\aB_n}$ has a quotient isomorphic to $\Z_2^3$. More precisely:  the abelianization of the point group $W_{\aB_n}$ is $\Z_2^2$. Hence, $L_2\rtimes W_{\aB_n}\twoheadrightarrow (L_2\rtimes W_{\aB_n})/L_1\cong L_2/L_1\rtimes W_{\aB_n}\twoheadrightarrow \Z_2 \times W_{\aB_n}^{\mathrm{ab}}\cong \Z_2^3$.

Note that the abelianization of $W_{\atB_n}$ is isomorphic to $\Z_2^2$.   Thus $W_{\atB_n}\cong L_1\rtimes W_{\aB_n}$. 

Since $W_{\atC_n}$  does not have a quotient isomorphic to $\Z_2^2\rtimes W_{\aB_n}$ or $\Z_4\rtimes W_{\aB_n}$ we know that $W_{\atC_n}\cong L_2\rtimes W_{\aB_n}$. Thus the groups $W_{\atB_n}\cong L_1\rtimes W_{\aB_n}$ and $W_{\atC_n}\cong L_2\rtimes W_{\aB_n}$ can be distinguished from $L_3\rtimes W_{\aB_n}$ by their finite quotients. 

Further, the abelianization of $W_{\atC_n}$ is $\Z_2^3$, thus the group $W_{\atB_n}$ can be distinguished from  $W_{\atC_n}$ by the abelianisation.  Finally, we obtain $G\cong W_{\atX_n}$.

It remains to deal with the case of a non-trivial direct product.   Let $W_{\Gamma_1},\ldots, W_{\Gamma_n}$ be irreducible affine Coxeter groups. Assume that $\widehat{W_{\Gamma_1}}\times\ldots\times\widehat{W_{\Gamma_n}}\cong\widehat{G}$. By \Cref{prop.invariants}, $G$ is a symmorphic crystallogrphic group. We may decompose $G$ as a direct product of directly indecomposable groups $G_1,\ldots, G_m$, thus $G\cong G_1\times\ldots\times G_m$ and each $G_i$ is a symmorphic crystallographic group for $i=1,\ldots, m$. Now, applying \cite[Proposition 2.17 (2)]{GrunewaldZalesskii2011} we obtain $n=m$ and that there exists $\sigma\in \sym(m)$ such that $\widehat{W_{\Gamma_i}}\cong \widehat{G_{\sigma(i)}}$. Since irreducible affine Coxeter groups are profinitely rigid we obtain $W_{\Gamma_i}\cong G_{\sigma(i)}$. Thus $W_{\Gamma_1}\times\ldots\times W_{\Gamma_m}\cong G$.
\end{proof}

\bibliographystyle{halpha}
\bibliography{refs.bib}

\begin{thebibliography}{CHMV23}

\bibitem[BCR16]{BridsonConderReid2016}
M.~R. Bridson, M.~D.~E. Conder, and A.~W. Reid.
\newblock Determining {F}uchsian groups by their finite quotients.
\newblock {\em Israel J. Math.}, 214(1):1--41, 2016.
\newblock \href{https://dx.doi.org/10.1007/s11856-016-1341-6}{{\ttfamily 10.1007/s11856-016-1341-6}}.

\bibitem[BMRS21]{BridsonMcReynoldsReidSpitler2021}
Martin~R. Bridson, D.~B. McReynolds, Alan~W. Reid, and Ryan Spitler.
\newblock On the profinite rigidity of triangle groups.
\newblock {\em Bull. Lond. Math. Soc.}, 53(6):1849--1862, 2021.
\newblock \href{https://dx.doi.org/10.1112/blms.12546}{{\ttfamily 10.1112/blms.12546}}.

\bibitem[Bri71]{Brigham1971}
Robert~C. Brigham.
\newblock On the isomorphism problem for just-infinite groups.
\newblock {\em Communications on pure and applied mathematics}, XXIV:789--796, 1971.

\bibitem[BRS23]{BridsonReidSpitler2023}
Martin~R. Bridson, Alan~W. Reid, and Ryan Spitler.
\newblock Absolute profinite rigidity, direct products, and finite presentability, 2023.
\newblock \href{https://arxiv.org/abs/2312.06058}{{\ttfamily arXiv:2312.06058}}.

\bibitem[CHMV23]{CorsonHughesMollerVarghese2023}
S.~M. Corson, S.~Hughes, P.~Möller, and O.~Varghese.
\newblock Higman-{T}hompson groups and profinite properties of right-angled {C}oxeter groups, 2023.
\newblock \href{https://arxiv.org/abs/2309.06213}{{\ttfamily arXiv:2309.06213 [math.GR]}}.

\bibitem[CR62]{CurtisReiner1962}
C.~W. Curtis and I.~Reiner.
\newblock {\em Representation Theory of finite groups and associative algebras}, volume~11 of {\em Pure and Applied Mathematics}.
\newblock Interscience Publishers, New Yourk, 1962.

\bibitem[Dav08]{Davis2008}
Michael~W. Davis.
\newblock {\em The geometry and topology of {C}oxeter groups}, volume~32 of {\em London Mathematical Society Monographs}.
\newblock Princeton University Press, 2008.

\bibitem[DFPR82]{Dixon1982}
J.~D. Dixon, E.~W. Formanek, J.~C. Poland, and L.~Ribes.
\newblock Profinite completions and isomorphic finite quotients.
\newblock {\em Journal of Pure and Applied Algebra}, 23:227--231, 1982.

\bibitem[Fei98]{Feit1998}
Walter Feit.
\newblock Integral representations of {W}eyl groups rationally equivalent to the reflection representation.
\newblock {\em J. Group Theory}, 1(3):213--218, 1998.
\newblock \href{https://dx.doi.org/10.1515/jgth.1998.013}{{\ttfamily 10.1515/jgth.1998.013}}.

\bibitem[FNP80]{Finken1980}
H.~Finken, J.~Neub\"{u}ser, and W.~Plesken.
\newblock Space groups and groups of prime-power order. {II}. {C}lassification of space groups by finite factor groups.
\newblock {\em Arch. Math. (Basel)}, 35(3):203--209, 1980.
\newblock \href{https://dx.doi.org/10.1007/BF01235339}{{\ttfamily 10.1007/BF01235339}}.

\bibitem[GS78]{GrunewaldSegal1978}
Fritz Grunewald and Daniel Segal.
\newblock Conjugacy in polycyclic groups.
\newblock {\em Comm. Algebra}, 6(8):775--798, 1978.
\newblock \href{https://dx.doi.org/10.1080/00927877808822268}{{\ttfamily 10.1080/00927877808822268}}.

\bibitem[GZ11]{GrunewaldZalesskii2011}
Fritz Grunewald and Pavel Zalesskii.
\newblock Genus for groups.
\newblock {\em J. Algebra}, 326:130--168, 2011.
\newblock \href{https://dx.doi.org/10.1016/j.jalgebra.2010.05.018}{{\ttfamily 10.1016/j.jalgebra.2010.05.018}}.

\bibitem[Kan01]{Kane2001}
Richard Kane.
\newblock {\em Reflection groups and invariant theory}, volume~5 of {\em CMS Books in Mathematics/Ouvrages de Math\'{e}matiques de la SMC}.
\newblock Springer-Verlag, New York, 2001.
\newblock \href{https://dx.doi.org/10.1007/978-1-4757-3542-0}{{\ttfamily 10.1007/978-1-4757-3542-0}}.

\bibitem[KW93]{KrophollerWilson93}
Peter~H. Kropholler and John~S. Wilson.
\newblock Torsion in profinite completions.
\newblock {\em J. Pure Appl. Algebra}, 88(1-3):143--154, 1993.
\newblock \href{https://dx.doi.org/10.1016/0022-4049(93)90018-O}{{\ttfamily 10.1016/0022-4049(93)90018-O}}.

\bibitem[Max98]{Maxwell1998}
George Maxwell.
\newblock The normal subgroups of finite and affine {C}oxeter groups.
\newblock {\em Proc. London Math. Soc. (3)}, 76(2):359--382, 1998.
\newblock \href{https://dx.doi.org/10.1112/S0024611598000112}{{\ttfamily 10.1112/S0024611598000112}}.

\bibitem[MV24]{MollerVarghese2023}
Philip M\"{o}ller and Olga Varghese.
\newblock On quotients of {C}oxeter groups.
\newblock {\em J. Algebra}, 639:516--531, 2024.
\newblock \href{https://dx.doi.org/10.1016/j.jalgebra.2023.09.048}{{\ttfamily 10.1016/j.jalgebra.2023.09.048}}.

\bibitem[PPW21]{PiwekPopovicWilkes2021}
Pawe\l\ Piwek, David Popovi\'{c}, and Gareth Wilkes.
\newblock Distinguishing crystallographic groups by their finite quotients.
\newblock {\em J. Algebra}, 565:548--563, 2021.
\newblock \href{https://dx.doi.org/10.1016/j.jalgebra.2020.06.033}{{\ttfamily 10.1016/j.jalgebra.2020.06.033}}.

\bibitem[PV24]{ParisVarghese2024}
Luis Paris and Olga Varghese.
\newblock Narrow normal subgroups of {C}oxeter groups and of automorphism groups of coxeter groups.
\newblock {\em Journal of Group theory}, 27:255--274, 2024.

\bibitem[Rei18]{Reid2018}
Alan~W. Reid.
\newblock Profinite rigidity.
\newblock In {\em Proceedings of the {I}nternational {C}ongress of {M}athematicians---{R}io de {J}aneiro 2018. {V}ol. {II}. {I}nvited lectures}, pages 1193--1216. World Sci. Publ., Hackensack, NJ, 2018.

\bibitem[RS22]{SantosRegoSchwer2022}
Yuri~Santos Rego and Petra Schwer.
\newblock The galaxy of {C}oxeter groups, 2022.
\newblock \href{https://arxiv.org/abs/2211.17038}{{\ttfamily arXiv:2211.17038 [math.GR]}}, arXiv:2211.17038 [math.GR].

\bibitem[RT10]{RatcliffeTschantz2010}
John~G. Ratcliffe and Steven~T. Tschantz.
\newblock Fibered orbifolds and crystallographic groups.
\newblock {\em Algebr. Geom. Topol.}, 10(3):1627--1664, 2010.
\newblock \href{https://dx.doi.org/10.2140/agt.2010.10.1627}{{\ttfamily 10.2140/agt.2010.10.1627}}.

\bibitem[Sch02]{Schur1902}
I.~Schur.
\newblock Neuer beweis eines satzes über endliche gruppen.
\newblock {\em Sitzungsberichte der Königlich Preussischen Akademie der Wissenschaften}, pages 1013--1019, 1902.

\end{thebibliography}
\end{document}